\documentclass[12pt]{article}
\usepackage[utf8]{inputenc}
\usepackage{amsmath}
\usepackage[english]{babel}
\usepackage{url}

\usepackage[a4paper,top=3cm,bottom=2cm,left=3cm,right=3cm,marginparwidth=1.75cm]{geometry}

\usepackage{amsmath}
\usepackage{graphicx}
\usepackage[colorinlistoftodos]{todonotes}
\usepackage[colorlinks=true, allcolors=blue]{hyperref}
\usepackage{amssymb}
\usepackage{amssymb}
\usepackage{amsthm}
\usepackage{authblk}
\usepackage{amsmath}
\usepackage{amsfonts}
\usepackage{amssymb}
\usepackage{amscd}
\usepackage{latexsym}
\usepackage{breqn}

\theoremstyle{plain}
\newtheorem{thm}{Theorem}[section]
\newtheorem{lem}[thm]{Lemma}
\newtheorem{prop}[thm]{Proposition}
\newtheorem{cor}[thm]{Corollary}

\theoremstyle{definition}
\newtheorem{conj}{Conjecture}[section]

\theoremstyle{remark}
\newtheorem*{rem}{Remark}

\title{Symbols of compact truncated Toeplitz operators}
\author{Ryan O'Loughlin \\ E-mail address: \href{mailto:R.OLoughlin@leeds.ac.uk}{R.OLoughlin@leeds.ac.uk}}
\affil{School of Mathematics, University of Leeds, Leeds, LS2 9JT, U.K.}
\date{}
\begin{document}
\maketitle
\begin{abstract}
This paper characterises the dual of the model space $K_I^1$, where $I$ is an inner function, intersected with the shifted Hardy space, $z H^1$. With this duality result, it is then shown that every bounded truncated Toeplitz operator on the model space $K_I^2$ has a bounded symbol if and only if every compact truncated Toeplitz operator on $K_I^2$ has a symbol which is of the form $I$ multiplied by a continuous function.

\end{abstract}

\maketitle

\section{Introduction}
In this paper we provide two new overlapping results. We first characterise the dual space of any given backward shift invariant subspace of the Hardy space $H^1$ intersected with  $ z H^1$. We then use our duality result to study the question: when does a bounded truncated Toeplitz operator have a bounded symbol? This question has generated much research interest and is one of the most fundamental problems concerning truncated Toeplitz operators. The main result of this paper shows that every bounded truncated Toeplitz operator on $K_I^2$ has a bounded symbol if and only if every compact truncated Toeplitz operator on $K_I^2$ has a symbol which is of the form $I \phi$ where $\phi$ is a continuous function on the unit circle.

Throughout we let $1 \leqslant p < \infty $ and we denote the unit circle by $\mathbb{T}$. All $L^p$ spaces, and $L^{\infty}$, are assumed to be on the unit circle, $\mathbb{T}$. We denote the Hardy space on the unit circle by $H^p$ and we refer the reader to \cite{duren1970theory, nikolski2002operators} for background theory on the Hardy space. Throughout we let $I$ be an inner function (i.e. a function in $H^2$ which is unimodular on $\mathbb{T}$). It is shown in Chapter 5 of \cite{cima2000backward} that all backward shift invariant subspaces of $H^p$ are of the form $K_I^p: = H^p \cap I \overline{H^p_0}$, where $\overline{H^p_0} := \{ \overline{f} : f \in z H^p \}$. We call the space $K_I^p$ a model space, and we refer the reader to \cite{cima2000backward} for background theory on model spaces. Integrals will be evaluated with respect to $m$, where $m$ is the normalised Lebesgue measure on $\mathbb{T}$. We define $C ( \mathbb{T})$ to be the continuous functions on the unit circle.

We let BMOA denote the set of all analytic functions of bounded mean oscillation, i.e., $ f \in \text{BMOA}$ means $f \in H^2$ and 
$$
\sup_A \frac{1}{|A|} \int_A |f - f_A | dm < \infty ,
$$
where the supremum is taken over all arcs $A \subset \mathbb{T}$ and 
$$
f_A := \frac{1}{|A|} \int_A f dm .
$$
It can be checked that BMOA is a linear vector space and an easy adaptation of Proposition 2.5 in \cite{BMOAbook} shows that when equipped with the norm
$$
\| f \|_{*} := \left| \int_{\mathbb{T}} f dm \right| + \sup_A \frac{1}{|A|} \int_A |f - f_A | dm ,
$$
BMOA becomes a Banach space. 

When $1< p < \infty$ we let $P: L^p \to H^p$ be the Riesz projection, and $P_{I} = P I ( Id - P)\overline{I}$ be the projection from $L^p$ to $K_I^p$. When $p = 2$, $P$ and $P_I$ are in fact orthogonal projections. As noted in Section 2 of \cite{oloughlin2020matrixvalued}, these projections can be expressed as integral operators independent of $p \in ( 1 , \infty)$. With these projections, for $1 < p < \infty$ we can write $L^p = \overline{H^p_0} \oplus H^p$ or $L^p = \overline{H^p_0} \oplus K_I^p \oplus I H^p$ (see page 20 of \cite{mythesis} for details) and furthermore for $p=2$ the summations are orthogonal summations.

For $\phi \in L^2$, we define the \textit{truncated Toeplitz operator} $A_{\phi}$ on bounded functions from $K_I^2$ by 
$$
A_{\phi}(f) = P_I ( \phi f ), \hskip 2cm f \in K_{I}^2 \cap L^{\infty},
$$
where here $P_I : L^2 \to K_I^2$. In contrast to the Toeplitz operators on $H^2$ the operator $A_{\phi}$ may be extended to a bounded operator on $K_I^2$ even for some unbounded symbols.

We use the notation $\mathcal{T}(I)$ to denote the space of bounded truncated Toeplitz operators on $K_I^2$ and $\mathcal{T}_c(I)$ to denote the space of compact truncated Toeplitz operators on $K_I^2$.
Previous results in \cite{dualofmodelspacep1} identify the dual space of $K_I^1 \cap z H^1$ for a certain class of inner functions. 

The results of Section 2 give an alternative description of the space dual to $K_I^1 \cap z H^1$, and furthermore this description is valid for all inner functions. Specifically, the description of the space dual to $K_I^1 \cap z H^1$ given in Section 2 initially realises $(K_I^1 \cap z H^1)^*$ as a quotient space and then we further show that $(K_I^1 \cap z H^1)^*$ may be realised as a space of analytic functions.

In Section 3 we will see that this realisation of $(K_I^1 \cap z H^1)^*$ lays the foundation for us to prove the main result of this paper, which is every bounded truncated Toeplitz operator on $K_I^2$ has a bounded symbol if and only if every compact truncated Toeplitz operator on $K_I^2$ has a symbol lying in $I C ( \mathbb{T})$.
\section{Duality results}

In this section we describe the dual and predual of the space $K_I^1 \cap z H^1$. This description is a crucial tool needed in Section 2, where we study the symbols of truncated Toeplitz operators.

For not necessarily closed subspaces $M \subseteq X$ and $N \subseteq X^*$ of a Banach space $X$ we define the annihilator
$$
\mathcal{M}^{\perp}:=\left\{\ell \in X^{*}: \ell(x)=0 \quad \forall x \in \mathcal{M}\right\}
$$
and the pre-annihilator
$$
^\perp \mathcal{N}:=\{x \in X: \ell(x)=0 \quad \forall \ell \in \mathcal{N}\}.
$$
We note that with the above definitions the Hahn-Banach separation Theorem implies that $^\perp(M^{\perp})$ is the norm closure of $M$.

\begin{lem}\label{dense}
$K_I^2 \cap z H^2$ is dense in $K_I^1 \cap z H^1$.
\end{lem}
\begin{proof}
Clearly we have $I \overline{z H^2_0} \cap H^2 \subset I \overline{z H^1_0} \cap H^1$. If we show 
$$
(I \overline{z H^2_0} \cap H^2)^{\perp} \subseteq (I \overline{z H^1_0} \cap H^1)^{\perp},
$$
then as $^\perp(M^{\perp})$ is the norm closure of $M$ for a subspace $M$ we obtain that  $I \overline{ z H^2_0 } \cap H^2$ is dense in $I \overline{ z H^1_0 } \cap H^1$.

A result originally due to Fefferman-Stein \cite{fefferman1972h}, which can also be found as Theorem 2.2 in Chapter 9 of \cite{topicsincomlpexanalysis}, shows that the dual space of $H^1$ is BMOA. By duality in the $H^2$ sense (which agrees with the $H^1$-BMOA duality pairing) we have if $g \in (I \overline{z H^2_0} \cap H^2)^{\perp}$ then $ g \in \text{BMOA}$ and $\langle I\overline{z H^2_0} , g \rangle = 0 $, which implies $ g \in \frac{I}{z}H^2 \cap \text{BMOA}$ so $g z \in \text{BMOA} \cap I H^2$. If we define $G:= \frac{gz}{I}$, then as the space $H^1$ is invariant by multiplication by $I$, we can deduce that BMOA (being the dual of $H^1)$ is invariant under the Toeplitz operator $T_{\overline{I}}$ and so $G \in \text{BMOA}$. If $f \in I \overline{z H^1_0} \cap H^1 $ then $\frac{fz}{I}= \overline{F}$ for some $F \in H^1_0$, and if we denote $\langle \hskip 0.1cm , \hskip 0.1cm \rangle $ to be the $H^1$-BMOA duality pairing then we have 
$$
\langle f,g \rangle = \langle \frac{fz}{I} , \frac{gz}{I} \rangle = \langle \overline{F} , G \rangle =0.
$$
Thus, we have established $I \overline{ z H^2_0 } \cap H^2$ is dense in $I \overline{ z H^1_0 } \cap H^1$. Now clearly as multiplication by $z$ is an isometry we have $K_I^2 \cap z H^2$ is dense in $K_I^1 \cap z H^1$.

\end{proof}

\begin{thm}\label{BMOA1dual}
$l \in (K_I^1 \cap z H^1)^* $ if and only if there is a unique $ v + Q \in L^{\infty} / Q$, where $Q = L^{\infty} \cap (H^2 + \overline{I H^2}) $, such that
$$
l (f) = l_{v+Q} (f) : = \int_{\mathbb{T}} f({\zeta}) {v( \zeta)} \hskip 0.1cm dm(\zeta).
$$
Furthermore $\|l_{v+Q} \| = \|v + Q \|_{L^{\infty}/Q}$.
\end{thm}

\begin{proof}
Clearly as $v \in L^{\infty}$ this will define a bounded linear functional. This functional is well defined as if $v_2$ is another representative of $v$ then $v - v_2 \in Q$ but for all $f$ in $K_I^2 \cap z H^2$ and any $q \in Q$ by orthogonality in the $H^2$ sense we have 
$\int_{\mathbb{T}} f({\zeta}) {q( \zeta)} \hskip 0.1cm dm(\zeta) = 0$ and so by the above lemma we must indeed have $\int_{\mathbb{T}} f({\zeta}) ({v( \zeta) - v_2 ( \zeta)}) \hskip 0.1cm dm(\zeta) = 0$ for all $f \in K_I^1 \cap z H^1$.

To prove the converse implication we must find the annihilator of $K_I^1 \cap z H^1$. The annihilator of $K_I^1 \cap z H^1$ (viewed as a subspace of $L^1$) is given by  $(K_I^1 \cap z H^1)^{\perp} = \{ b \in L^{\infty} : \int_{\mathbb{T}} f({\zeta}) {b( \zeta)} \hskip 0.1cm dm(\zeta) = 0 \text{ for all } f \in K_I^1 \cap z H^1 \}$, but as a result of our previous lemma this is equal to $\{ b \in L^{\infty} : \int_{\mathbb{T}} f({\zeta}) {b( \zeta)} \hskip 0.1cm dm(\zeta) = 0 \text{ for all } f \in K_I^2 \cap z H^2 \}$.

We now argue that $$\{ b \in L^{\infty} : \int_{\mathbb{T}} f({\zeta}) {b( \zeta)} \hskip 0.1cm dm(\zeta) = 0 \text{ for all } f \in K_I^2 \cap z H^2 \} = Q:= L^{\infty} \cap (H^2 + \overline{I H^2}).$$

A straightforward computation verifies the $\supseteq$ inclusion. Before we show the $\subseteq$ inclusion we first note that $f \in K_{I}^{2} \cap z H^{2}$ if and only if $f \in K_{I}^{2}$ and $f(0)=0$, which is equivalent to $f \in K_{I}^{2}$ and $f \perp k_{0}^{I} $, where $1-\overline{I(0)} I(z) = k_{0}^{I}(z) \in K_I^2$ is the reproducing kernel at $0$, (see Chapter 5 of \cite{modelspacesgarcia} for details on the reproducing kernels for model spaces). Hence $K_{I}^{2} \cap z H^{2}=K_{I}^{2} \ominus \mathbb{C} k_{0}^{I}$, giving $K_{I}^{2} \ominus\left(K_{I}^{2} \cap z H^{2}\right)=\mathbb{C} k_{0}^{I} .$ Now to prove the $\subseteq$ inclusion, take $b \in L^{\infty}$ such that $\langle f, \bar{b}\rangle_{2}=0$ for every $f \in K_{I}^{2} \cap z H^{2} .$ Decompose $b=g+\overline{I \varphi_{1}+\varphi_{2}}$, with $g, \varphi_{1} \in H^{2}$ and $\varphi_{2} \in K_{I}^{2} .$ Hence, for every $f \in K_{I}^{2} \cap z H^{2}$, we have
$$
0=\langle f, \bar{g}\rangle_{2}+\left\langle f, I \varphi_{1}\right\rangle_{2}+\left\langle f, \varphi_{2}\right\rangle_{2}.
$$
The first term $\langle f, \bar{g}\rangle_{2}$ is zero because $f \in z H^{2}$ and $g \in H^{2}$, the second term $\left\langle f, I \varphi_{1}\right\rangle_{2}$ is also zero because $f \in K_{I}^{2}$ and $I \varphi_{1} \in I H^{2}$. Thus we obtain that, for every $f \in K_{I}^{2} \cap z H^{2},\left\langle f, \varphi_{2}\right\rangle_{2}=0$. That means that $\varphi_{2} \in K_{I}^{2} \ominus\left(K_{I}^{2} \ominus z H^{2}\right)$, and as previously pointed out $K_{I}^{2} \ominus\left(K_{I}^{2} \cap z H^{2}\right)=\mathbb{C} k_{0}^{I}$, which means $\varphi_{2}=c(1-\overline{I(0)} I)$ for some $c \in \mathbb{C}$. Hence
$$
b=g+\overline{I \varphi_{1}}+\overline{c(1-\overline{I(0)} I)}=g+\bar{c}+\overline{I\left(\varphi_{1}-c \overline{I(0)}\right)}.
$$
We then deduce that $b \in L^{\infty} \cap\left(H^{2}+\overline{I H^{2}}\right)=Q .$



Now given $l \in (K_I^1 \cap z H^1)^* $, the Hahn-Banach Theorem shows there is a unique ${v + Q}$ where $ v \in L^{\infty}$ such that $l(f) =l_{v + Q} (f) =  \int_{\mathbb{T}} f({\zeta}) {v( \zeta)} \hskip 0.1cm dm(\zeta)$ and furthermore that $\|l_{v+Q} \| = \|v + Q \|_{L^{\infty}/Q}$. We refer the reader to page 97 Theorem 4.9 of \cite{rudinsfunanal} for details.


\end{proof}

One can also obtain a set theoretic description of $(K_I^1 \cap z H^1)^*$ which more so resembles the classical duality result for model spaces, which is $(K_I^p)^* =  K_I^q$  where $1 < p < \infty$ and $\frac{1}{p} + \frac{1}{q} = 1$ (see Theorem 5.10.1 in \cite{cima2000backward}).

We let $P_{I,z}$ denote the orthogonal projection $L^2 \to K_I^2 \cap z H^2 $ and we let $E:L^{\infty}/Q \to P_{I,z}(L^{\infty})$ be the well defined map $v + Q \mapsto P_{I,z}( \overline{v})$. If we equip $P_{I,z}(L^{\infty})$ with the norm $\| P_{I,z}(f) \| := \inf \{ \|g \|_{L^{\infty}} : P_{I,z}(g) = P_{I,z}(f) \}$ then $E$ is an isometric antilinear isomorphism. Now as $P_{I,z}P = P_{I,z}$ and $P(L^{\infty}) =  \text{BMOA}$ (see Theorem 3.5.11 in \cite{cima2000backward}), we may view $E$ as a map from $L^{\infty}/ Q $ to $P_{I,z}(\text{BMOA})$. Furthermore, we can write $P_{I,z}( \text{BMOA})$ as 
$$
\{ k \in K_I^2 \cap z H^2 :  \exists h \in H^2 \text{ and } \exists c \in K_I^2   \ominus (K_I^2 \cap z H^2) \text{ with } c+  k + Ih \in \text{BMOA} \} := K.
$$
Noting that, as in the proof of Theorem \ref{BMOA1dual}, we have that $K_I^2   \ominus (K_I^2 \cap z H^2) = \mathbb{C} k_0^I \subseteq H^{\infty} \subseteq \text{BMOA}$, this allows us to further write $K$ as 
$$
\{ k \in K_I^2 \cap z H^2 :  \exists h \in H^2  \text{ with }   k + Ih \in \text{BMOA} \}.
$$

As mentioned previously, because the space $H^1$ is invariant by multiplication by $I$, and as BMOA is the dual space of $H^1$ (see Theorem 2.2 in Chapter 9 of \cite{topicsincomlpexanalysis}), we can deduce that BMOA is invariant under the Toeplitz operator $T_{\overline{I}}$. So, if $k \in K$ is such that $k + I h = b \in \text{BMOA}$ with $h \in H^2$, then $T_{\overline{I}}( k + I h ) = T_{\overline{I}}(b)$, and so as $k \in \ker T_{\overline{I}}$, we have $ h = T_{\overline{I}}(b) \in \text{BMOA}$. Thus, we can in fact write $K$ as
$$
\{ k \in K_I^2\cap z H^2 :  \exists h \in \text{BMOA}\text{ with }   k + Ih \in \text{BMOA} \}.
$$
The above description of $K$ is clearly equal to $\text{span}(\text{BMOA}, I (\text{BMOA}) ) \cap K_I^2 \cap z H^2$, so we conclude
$$
P_{I, z}( L^{\infty}) = \text{span}(\text{BMOA}, I (\text{BMOA})) \cap K_I^2\cap z H^2 .
$$

Thus via $E$ and the above expression for $P_{I, z}( L^{\infty})$ we can isometrically identify $(K_I^1 \cap z H^1)^*$ as $\text{span}(\text{BMOA}, I (\text{BMOA})) \cap K_I^2\cap z H^2 $, where for $k \in K$ we have  $\| k\| = \inf \{ \|g \|_{L^{\infty}} : P_{I,z}(g) = k \}$. 

\begin{rem}
We remark that adapting the above theorem and reasoning we can similarly deduce $ (K_I^1)^* = P_{I}( \text{BMOA}) = \text{span}(\text{BMOA}, I (\text{BMOA})) \cap K_I^2 $ where for $k \in \text{span}(\text{BMOA}, I (\text{BMOA})) \cap K_I^2 $ we have  $\| k\| = \inf \{ \|g \|_{L^{\infty}} : P_{I}(g) = k \}$. 

We note that in general $I \text{BMOA} \not\subseteq \text{BMOA}$. In fact, the conditions for when $I \text{BMOA} \subseteq \text{BMOA}$ can be found as Theorem 1 in \cite{multiplybmoa}.
\end{rem}

If one wants to consider the predual of the space $K_I^1 \cap z H^1$, we have the following result which may be found as Lemma 3.1 in \cite{bessenovfredholm}. 

\begin{thm}\label{baranovpredual}
Let $\mathcal{F}_I$ be the closure of the set $\overline{I H^{\infty}} + H^{\infty}$ in the weak $*$ topology of the space $L^{\infty}$. Then $(C ( \mathbb{T}) / (\mathcal{F}_I \cap C ( \mathbb{T}) ))^* = K_I^1 \cap z H^1  $.
\end{thm}
Specifically, in the above theorem there is an isometric isomorphism given by $h \mapsto l_h$ where $h \in K_I^1 \cap z H^1$ and $l_h \in (C ( \mathbb{T}) / (\mathcal{F}_I \cap C ( \mathbb{T}) ))^*$ is defined by $l_h ( \phi )  = \int_{\mathbb{T}} \phi h \hskip 0.1cm dm $ where $\phi $ is any representative of an element in $(C ( \mathbb{T}) / (\mathcal{F}_I \cap C ( \mathbb{T}) ))$.

\section{Application to truncated Toeplitz operators}

In this section after giving a brief overview on some of the previously known results about the symbols of truncated Toeplitz operators, use our duality results from our previous section to deduce the main result of this paper, which is Theorem \ref{XX}.

The question of whether every bounded truncated Toeplitz operator has a bounded symbol is an interesting one. This question has led to much research activity within the community with many questions being answered and many new questions being posed. Here we give the reader a brief background on this topic. In Sarason's seminal work of 2007 \cite{sarason2007algebraic} he initiated a systematic study of truncated Toeplitz operators with symbols in $L^2$. In this paper, one of the most natural questions posed was whether every bounded truncated Toeplitz operator has a bounded symbol. This question was then shown to be negative in 2009 (see \cite{baranov2010bounded}). In fact, the authors actually constructed a bounded rank one truncated Toeplitz operator which was shown to have no bounded symbol. To build on this work, in \cite{baranov2010symbols} the authors gave two conditions on the inner function, $I$, which are equivalent to every bounded truncated Toeplitz operator on $K_I^2$ having a bounded symbol. (See Theorem \ref{3.2} below.)

Motivated by these findings, a similar study in to the symbols of compact truncated Toeplitz operators was initiated. Section 5 of \cite{Isabellesurvey} gives an overview of many results in this area. In particular, the role played by bounded symbols in the case of bounded truncated Toeplitz operators seems to be replaced by symbols of the form $I C( \mathbb{T})$ when we are considering compact truncated Toeplitz operators. Specifically, Proposition 5.4 of \cite{Isabellesurvey} shows that if $\phi \in I C(\mathbb{T})$ then $A_{\phi}^I$ is compact, however using the example of the rank one truncated Toeplitz operator which does not possess a bounded symbol, we know there exists a compact truncated Toeplitz operator without a symbol in $I C ( \mathbb{T})$, as trivially $I C ( \mathbb{T}) \subseteq L^{\infty}$. One question posed in \cite{Isabellesurvey} was whether there was a compact truncated Toeplitz operator with a symbol in $I C(\mathbb{T}) + I H^{\infty}$ that has no continuous symbol. This question was then answered affirmatively in \cite{pamJRP} when a compact truncated Toeplitz operator with this property was then constructed. 

Following the results in \cite{baranov2010symbols} one may suspect that there are conditions on the inner function $I$ which are equivalent to every compact truncated Toeplitz operator on $K_I^2$ having a symbol in $I C( \mathbb{T})$. We may further suspect that these conditions may be similar in nature to the conditions on the inner function $I$ which are equivalent to every bounded truncated Toeplitz operator on $K_I^2$ having a bounded symbol. In fact, in this section we will prove that every compact truncated Toeplitz operator on $K_I^2$ has a symbol in $I C( \mathbb{T})$ if and only if every bounded truncated Toeplitz operator on $K_I^2$ has a bounded symbol. We show this with Theorem \ref{XX} below.

In the following we define $\mathcal{C}_p (I)$ to be the set of all finite complex Borel measures $\mu$ on the unit circle such that the embedding $K_I^p \to L^p (|\mu | )$ is continuous.

\begin{thm}\label{3.2} \cite{baranov2010symbols}
\newline
The following are equivalent:
\begin{enumerate}
    \item any bounded truncated Toeplitz operator on $K_I^2$ admits a bounded symbol;
    \item $\mathcal{C}_1 (I^2) = \mathcal{C}_2 (I^2)$;
    \item for any $f \in H^1 \cap I^2 \overline{z H^1_0}$ there exists $x_i, y_i \in K_I^2$ with $ \sum_i \|x_i \|_2 \|y_i\|_2 < \infty$ such that $f= \sum_i x_i y_i$.
\end{enumerate}
\end{thm}

The inner function $I$ is said to be one-component if and only if there exists an $\eta \in (0,1)$ such that
$$
\{ z \in \mathbb{D} : | I(z) | < \eta \}
$$
is connected. We remark that by Corollary 2.5 in \cite{baranov2010symbols} the equivalent conditions of the theorem below are fulfilled when $I$ is a one component inner function.

\vskip 0.4cm
\textbf{Throughout the rest of the paper we will focus on proving the following theorem.}

\vskip 0.2cm

\begin{thm}\label{XX}
The equivalent conditions of Theorem \ref{3.2} are satisfied if and only if any compact truncated Toeplitz operator on $K_I^2$ has a symbol in $I C ( \mathbb{T})$.
\end{thm}
\vskip 0.4cm
We note that the forward implication of the above theorem has already been noticed. Indeed, as noted in Section 3.2 of \cite{Isabellesurvey}, we have $\mathcal{C}_2 (I^2) = \mathcal{C}_2 (I)$, and thus in view of Theorem \ref{3.2}, Theorem 5.2 in \cite{Isabellesurvey} states if every bounded truncated Toeplitz operator on $K_I^2$ has a bounded symbol then every compact truncated Toeplitz operator on $K_I^2$ is of the form $A_{I \phi}$ where $\phi \in C ( \mathbb{T})$.

The main theorem of this paper is the backwards implication of the above theorem, that is, if any compact truncated Toeplitz operator on $K_I^2$ has a symbol in $I C ( \mathbb{T})$, then the equivalent conditions of Theorem \ref{3.2} are satisfied. As the workings for the proof of the backwards implication of Theorem \ref{XX} are quite long and detailed we first outline the idea of the proof for the reader.

We first consider the map $L: X \to  K_{I^2}^1 \cap z H^1,$ given by $f \mapsto  I f $ (where $X$ is defined as in \eqref{Xdefn} below). Then using the results of Section 1 as well as Theorems \ref{cptduality} and \ref{dualTTO}, which identify the predual and dual space of $X$ as $\mathcal{T}_c(I)$ and $\mathcal{T}(I)$ respectively, we can describe the pre-adjoint, $^*L$, and adjoint, $L^*$, of $L$,
\begin{align*}
    ^*L:& C ( \mathbb{T}) / (\mathcal{F}_{I^2} \cap C ( \mathbb{T}) ) \to \mathcal{T}_c  (I), \quad &&g+\mathcal{F}_{I^2} \cap C ( \mathbb{T})  \mapsto A_{I f }^I \\
    L:& X \to z H^1 \cap K_{I^2}^1, \quad &&f \mapsto  I f \\
    L^*:& L^{\infty} / Q_2 \to \mathcal{T}  (I), \quad &&g+Q_2 \mapsto A_{I f}^I,
\end{align*}
where $Q_2  = L^{\infty} \cap (H^2 + \overline{I^2 H^2}) $. The above description of $L^*$ is given by Theorem \ref{Lstar} and the description of $^*L$ is given by Proposition \ref{starS}.

Then for an inner function $I$, by Proposition \ref{propcompact} we notice that any compact truncated Toeplitz operator on $K_I^2$ having a symbol in $I C ( \mathbb{T})$, is actually equivalent to $^*L$ being isomorphic. Then when $^*L$ is isomorphic, $L^*$ must also be isomorphic and in particular the image of $L^*$ must be equal to $\mathcal{T}(I)$. Then under the assumption $^*L$ is surjective, by Corollary \ref{boundedimage} (which shows the image of $L^*$ is all truncated Toeplitz operators on $K_I^2$ with a bounded symbol) we are able to deduce that every bounded truncated Toeplitz operator on $K_I^2$ has a bounded symbol. 
\vskip 0.6cm
Following the results of \cite{baranov2010symbols}, we define the Banach spaces
\begin{equation}\label{Xdefn}
    X = \{ \sum x_i \overline{y_i} : x_i , y_i \in K_I^2 , \sum \|x_i\|_2 \|y_i \|_2  < \infty \}, 
\end{equation}
and
$$
X_a = \{ \sum x_i y_i : x_i , y_i \in K_I^2 , \sum \|x_i\|_2 \|y_i \|_2  < \infty \}. 
$$
The norm in the space of $X$ and $X_a$ is defined as the infimum of $ \sum \|x_i\|_2 \|y_i \|_2$ over all possible representations. We note there is an isometric isomorphism from $X$ to $X_a$ given by 
\begin{equation}\label{XtoXa}
    f \mapsto \overline{z} I f,
\end{equation}
and one can also show that the forward shift from $X_a \to K_{I^2}^1 \cap z H^1$ is bounded. The following theorem can be found as Theorem 2.3 in \cite{baranov2010symbols}.

\begin{thm}\label{dualTTO}
The dual space of $X$ can be naturally identified with $\mathcal{T}(I)$. Namely, continuous linear functionals over $X$ are of the form 
$$
\Phi_A (f) = \sum_i \langle A x_i , y_i \rangle , \hskip 2cm f = \sum_i x_i \overline{y_i} \in X,
$$
with $A \in \mathcal{T}(I)$, and the correspondence between $X$ and $\mathcal{T}(I)$ is one to one and isometric.

\end{thm}

We define a bounded linear map
$$L: X \to  K_{I^2}^1 \cap z H^1,$$
given by 
$$
f \mapsto  I f .
$$
Now taking into account the results of Section 1 and Theorem \ref{dualTTO}, when considering the adjoint, $L^*$, of $L$ we obtain a bounded map
$$
L^* : L^{\infty} / Q_2 \to \mathcal{T}(I),
$$
where $Q_2  = L^{\infty} \cap (H^2 + \overline{I^2 H^2}) $ Explicitly, $L^*$ is the unique map satisfying
\begin{equation}\label{equate}
    \langle L(k) , g + Q_2 \rangle = \langle k , L^*(g + Q_2) \rangle,
\end{equation}
for each $k \in X$ and $g + Q_2 \in  L^{\infty} / Q_2$ (here the duality pairings, denoted $\langle \hskip 0.1cm, \rangle $, are given by Theorem \ref{BMOA1dual} and \ref{dualTTO} respectively). If we denote $A$ to be the truncated Toeplitz operator $L^*(g + Q_2)$, then for $k = x\overline{y}$ with $x \in K_I^{\infty}, y \in K_I^2$ equating both sides of equation \eqref{equate} gives us 
$$
\int_{\mathbb{T}} x \overline{y} I g \hskip 0.2cm dm = \int_{\mathbb{T}} A(x) \overline{y} \hskip 0.2cm dm ,
$$
and so as $y \in K_I^2$, and by orthogonality we know $\langle f, y \rangle_{2} = \langle P_{I}(f), y \rangle_{2} $ for each $f \in L^2$ we have
$$
\int_{\mathbb{T}} (A(x) - x I g) \overline{y} \hskip 0.2cm dm  = \int_{\mathbb{T}} (A(x) - P_{I}(x I g)) \overline{y} \hskip 0.2cm dm = 0.
$$
Now by density of $K_I^{\infty}$ in $K_I^2$ (see Section 2.5 of \cite{recentprogressGarcia}) and non-degeneracy of the integral we can deduce
$$
A = A_{I g}^I.
$$
We conclude the following result;

\begin{thm}\label{Lstar}
There is a bounded linear map $ L^*: L^{\infty} / Q_2 \to \mathcal{T}(I)$, given by 
\begin{equation}\label{bounded}
    g + Q_2 \mapsto A_{Ig}.
\end{equation}
\end{thm}

\begin{cor}\label{boundedimage}
The image of $L^*$ is exactly all elements of $\mathcal{T}(I)$ which possess a bounded symbol.
\end{cor}
\begin{proof}
Clearly the image of $L^*$ is contained the set of all bounded truncated Toeplitz operators with a bounded symbol. Conversely, if $A_{\phi}^I$ has a bounded symbol $g$ then $\overline{I} g\in L^{\infty}$, so $\overline{I} g + Q_2 \in L^{\infty} / Q_2$ and maps to $A_{\phi}^I$ through $L^*$.
\end{proof}
In the case when every bounded truncated Toeplitz operator has a bounded symbol the map $L^*$ is isomorphic, and so the norm of $A_{Ig}$ is equivalent to $\| g + Q_2 \|_{L^{\infty}/Q_2}.$

We now seek to describe the pre-adjoint of $L$. The second part of Theorem 2.3 in \cite{baranov2010symbols} states the following.
\begin{thm}\label{cptduality}
The dual space of all compact truncated Toeplitz operators can be identified with $X$, via the duality pairing $\sum_i x_i \overline{y_i}  \mapsto L_{\sum_i x_i \overline{y_i} }$, where
$$
L_{\sum_i x_i \overline{y_i} } (A) = \sum_i \langle A x_i, y_i \rangle,
$$
for each compact truncated Toeplitz operator $A$. Furthermore the duality pairing $\sum_i x_i \overline{y_i}  \mapsto L_{\sum_i x_i \overline{y_i} }$ is one-to-one and isometric map between $X$ and $(\mathcal{T}_c(I))^*$.
\end{thm}
In general the pre-adjoint of a bounded linear map may not exist. Nonetheless we may define the map $^*L:C ( \mathbb{T}) / (\mathcal{F}_{I^2} \cap C ( \mathbb{T}) ) \to \mathcal{T}_c(I)$, where $$g + (\mathcal{F}_{I^2} \cap C ( \mathbb{T})) \mapsto A_{Ig}.$$ 
\begin{prop}\label{starS}
The map $^*L$ is a well defined bounded, linear, injective map and $(^*L)^* = L$ (i.e. $^*L$ is the pre-adjoint of $L$).
\end{prop}
\begin{proof}
The map $^*L$ is clearly linear. Proposition 5.4 in \cite{Isabellesurvey} ensures that the specified $A_{Ig}$ is indeed a compact truncated Toeplitz operator. Noting that $A_{\phi}^I =0$ if and only if $\phi \in \overline{I H^2} + I H^2$ (see Section 3 of \cite{sarason2007algebraic}), we see that the map is well defined and injective. By non-uniqueness of the symbol we have that 
$$
\| A_{Ig} \| \leqslant \inf \{ \|\overline{I h_1} + Ig + I h_2 \|_{L^{\infty}} \} = \inf \{ \|\overline{I^2 h_1} + g +  h_2 \|_{L^{\infty}} \} 
$$
taken over all $\overline{I^2 h_1} + h_2 \in  (\mathcal{F}_{I^2} \cap C ( \mathbb{T}) )$. We see the above expression is actually equal to $\| g + (\mathcal{F}_{I^2} \cap C ( \mathbb{T} )) \|$, and thus $^*L$ is bounded

We now argue $(^*L)^* = L$. Proposition 4.1 in \cite{baranov2010symbols} states that every element of $X$ can be expressed as a sum of four elements of the form $x \overline{y}$ for $x, y \in K_I^2$. So for an arbitrary element $\sum_{i=1}^4 x_i \overline{y_i} \in X$, we know $(^*L)^*$ is the linear map satisfying 
$$
\langle ^*L (g + (\mathcal{F}_{I^2} \cap C ( \mathbb{T}) )),\sum_{i=1}^4 x_i \overline{y_i} \rangle = \langle g +(\mathcal{F}_{I^2} \cap C ( \mathbb{T}) ) ,(^*L)^*(\sum_{i=1}^4 x_i \overline{y_i}) \rangle ,
$$
for every $g + (\mathcal{F}_{I^2} \cap C ( \mathbb{T}) ) \in C ( \mathbb{T}) / (\mathcal{F}_{I^2} \cap C ( \mathbb{T}) )$ and every $\sum_{i=1}^4 x_i \overline{y_i} \in X$. Here the duality pairing on the left hand side is understood by the duality described in Theorem \ref{cptduality} and on the right hand side the duality is described by Theorem \ref{baranovpredual}. Explicitly, this means $(^*L)^*$ is a linear map satisfying
\begin{equation}\label{setghere}
\int_{\mathbb{T}} I g \sum_{i=1}^4 x_i \overline{y_i} \hskip 0.1cm dm =  \int_{\mathbb{T}} g (^*L)^* (\sum_{i=1}^4 x_i \overline{y_i}) \hskip 0.1cm dm ,
\end{equation}
i.e.
\begin{equation}\label{starnotes}
    \int_{\mathbb{T}}g \underbrace{\left( I \sum_{i=1}^4 x_i \overline{y_i} - (^*L)^* (\sum_{i=1}^4 x_i \overline{y_i}) \right) \hskip 0.1cm}_{:=u} dm  = 0
\end{equation}
for all $g + (\mathcal{F}_{I^2} \cap C ( \mathbb{T}) )  \in C ( \mathbb{T}) / (\mathcal{F}_{I^2} \cap C ( \mathbb{T}) )$. Now, $u$ defined as above lies in $K_{I^2}^1 \cap z H^1$, so by Theorem \ref{baranovpredual} the map $l_u$, where $l_u (g + (\mathcal{F}_{I^2} \cap C ( \mathbb{T}) )) \mapsto \int_{\mathbb{T}} g u \hskip 0.1cm dm $ defines a continuous linear functional on $ C ( \mathbb{T}) / (\mathcal{F}_{I^2} \cap C ( \mathbb{T}) )$. So when \eqref{starnotes} holds we must have $l_u$ is the zero functional, however by Theorem \ref{baranovpredual} we know $u \mapsto l_u$ is an isometric isomorphism and so we must have $u = 0$ and so $ (^*L)^* (\sum_{i=1}^4 x_i \overline{y_i})= I \sum_{i=1}^4 x_i \overline{y_i}$ and therefore $ (^*L)^* = L$.

\end{proof}
We can make a result which is analogous to Corollary \ref{boundedimage} but in the case of continuous symbols.
\begin{prop}\label{propcompact}
The image of $^*L$ is all truncated Toeplitz operators of the form $A_{\phi}$ where $\phi \in I C ( \mathbb{T})$
\end{prop}

As the map $L$ is clearly injective, $^*L$ must have dense range and we can make the following corollary which has been previously noticed in the proof of Lemma 3.5 in \cite{bessenovfredholm}.

\begin{cor}
Truncated Toeplitz operators of the form $A_{\phi}$ where $\phi \in I C ( \mathbb{T})$ are dense in $\mathcal{T}_c(I)$.
\end{cor}

We now can prove our main result, which is a one way implication of Theorem \ref{XX}.
\begin{thm}\label{2}
If every compact truncated Toeplitz operator on $K_I^2$ is of the form $A_{I \phi}$ where $\phi \in C ( \mathbb{T})$ then every bounded truncated Toeplitz operator has a bounded symbol.
\end{thm}
\begin{proof}
If every compact truncated Toeplitz operator on $K_I^2$ is of the form $A_{I \phi}$ where $\phi \in C ( \mathbb{T})$ then by Proposition \ref{propcompact} we know that $^*L$ is surjective (and hence isomorphic). Now by Proposition \ref{starS} we must also have that $(^*L)^* = L$ is isomorphic, and hence $L^*$ is isomorphic. Now by Corollary \ref{boundedimage} this must mean every bounded truncated Toeplitz operator on $K_I^2$ has a bounded symbol.
\end{proof}

\begin{rem}
Examining the above proof we can obtain a concise alternate proof of a result already known from Theorem \ref{3.2}, which is $L$ is isomorphic (or equivalently condition 3 holds in Theorem \ref{3.2}) implies every bounded truncated Toeplitz operator on $K_I^2$ has a bounded symbol.
\end{rem}

We now easily state the proof of Theorem \ref{XX}.
\vskip 0.3cm
\noindent \textit{Proof of Theorem \ref{XX}}. The forward implication is already known by previous results (see the reasoning laid out after the statement of Theorem \ref{XX}) and the backwards implication is Theorem \ref{2}.
\vskip 0.5cm

A long standing open conjecture regarding symbols of bounded truncated Toeplitz operators is the following.
\begin{conj}\label{conj}
Every bounded truncated Toeplitz operator on $K_I^2$ has a bounded symbol if and only if $I$ is one-component.
\end{conj}
We can now state an equivalent formulation of the above conjecture, which is the following.
\begin{conj}
Every compact truncated Toeplitz operator on $K_I^2$ has a symbol in $I C ( \mathbb{T} )$ if and only if $I$ is one-component.
\end{conj}

\section*{Acknowledgements}
The author is grateful to EPSRC for financial support. 
\newline The author would like to thank Professor Partington for his guidance.
\newline The author is grateful to the referee for their valuable comments.

\section*{Declarations}
Declarations of interest: none.
\newline Competing interests: none.
\newline This research was supported by the EPSRC (grant number 1972662). 
\newline The sponsor (EPSRC) had no role in study design; in the collection, analysis and interpretation of data or in  writing of the report; and in the decision to submit the article for publication.

\bibliographystyle{plain}
\bibliography{bibliography.bib}

\begin{thebibliography}{10}

\bibitem{topicsincomlpexanalysis}
M.~Andersson.
\newblock {\em Topics in complex analysis}.
\newblock Springer New York, 1997.

\bibitem{baranov2010symbols}
A.~Baranov, R.~Bessonov, and V.~Kapustin.
\newblock Symbols of truncated {T}oeplitz operators.
\newblock {\em J. Funct. Anal.}, 261(12):3437--3456, 2011.

\bibitem{baranov2010bounded}
A.~Baranov, I.~Chalendar, E.~Fricain, J.~Mashreghi, and D.~Timotin.
\newblock Bounded symbols and reproducing kernel thesis for truncated
  {T}oeplitz operators.
\newblock {\em Journal of Functional Analysis}, 259(10):2673--2701, 2010.

\bibitem{dualofmodelspacep1}
R.~V. Bessonov.
\newblock Duality theorems for coinvariant subspaces of {$H^1$}.
\newblock {\em Adv. Math.}, 271:62--90, 2015.

\bibitem{bessenovfredholm}
R.~V. Bessonov.
\newblock Fredholmness and compactness of truncated {T}oeplitz and {H}ankel
  operators.
\newblock {\em Integral Equations Operator Theory}, 82(4):451--467, 2015.

\bibitem{Isabellesurvey}
I.~Chalendar, E.~Fricain, and D.~Timotin.
\newblock A survey of some recent results on truncated {T}oeplitz operators.
\newblock In {\em Recent progress on operator theory and approximation in
  spaces of analytic functions}, volume 679 of {\em Contemp. Math.}, pages
  59--77. Amer. Math. Soc., Providence, RI, 2016.

\bibitem{cima2000backward}
J.~A. Cima and W.~T. Ross.
\newblock {\em The backward shift on the Hardy space}.
\newblock American Mathematical Soc., 2000.

\bibitem{duren1970theory}
P.~L. Duren.
\newblock Theory of ${H}^p$ spaces.
\newblock {\em Pure Appl. Math}, 38:74, 1970.

\bibitem{multiplybmoa}
K.~M. Dyakonov.
\newblock Division and multiplication by inner functions and embedding theorems
  for star-invariant subspaces.
\newblock {\em American Journal of Mathematics}, 115(4):881--902, 1993.

\bibitem{fefferman1972h}
C.~Fefferman and E.~M. Stein.
\newblock $ {H}^{p} $ spaces of several variables.
\newblock {\em Acta mathematica}, 129:137--193, 1972.

\bibitem{modelspacesgarcia}
S.~R. Garcia, J.~Mashreghi, and W.~T. Ross.
\newblock {\em Introduction to model spaces and their operators}, volume 148 of
  {\em Cambridge Studies in Advanced Mathematics}.
\newblock Cambridge University Press, Cambridge, 2016.

\bibitem{recentprogressGarcia}
S.R Garcia and W.T. Ross.
\newblock Recent progress on truncated {T}oeplitz operators.
\newblock {\em Blaschke Products and Their Applications}, 65:265--319, 2013.

\bibitem{BMOAbook}
D.~Girela.
\newblock Analytic functions of bounded mean oscillation.
\newblock In {\em Complex Function Spaces (Mekrijärvi, 1999) Univ. Joensuu
  Dept. Math. Rep. Ser.4}, pages 61--170. Univ. Joensuu, Joensuu, 2001.

\bibitem{pamJRP}
P.~Gorkin and J.~R. Partington.
\newblock Norms of truncated {T}oeplitz operators and numerical radii of
  restricted shifts.
\newblock {\em Comput. Methods Funct. Theory}, 19(3):487--508, 2019.

\bibitem{nikolski2002operators}
N.~K. Nikolski.
\newblock {\em Operators, Functions, and Systems-An Easy Reading: Hardy,
  Hankel, and Toeplitz}, volume~1.
\newblock American Mathematical Soc., 2002.

\bibitem{oloughlin2020matrixvalued}
R.~O'Loughlin.
\newblock Matrix-valued truncated {T}oeplitz operators: unbounded symbols,
  kernels and equivalence after extension.
\newblock {\em arXiv preprint arXiv:2012.00654}, 2020.

\bibitem{mythesis}
R.~O'Loughlin.
\newblock Multidimensional {T}oeplitz and truncated {T}oeplitz operators.
\newblock \url{https://etheses.whiterose.ac.uk/29284/}, University of {L}eeds,
  2021.

\bibitem{rudinsfunanal}
Walter Rudin.
\newblock {\em Functional analysis}.
\newblock International Series in Pure and Applied Mathematics. McGraw-Hill,
  Inc., New York, second edition, 1991.

\bibitem{sarason2007algebraic}
D.~Sarason.
\newblock Algebraic properties of truncated {T}oeplitz operators.
\newblock {\em Oper. Matrices}, 1(4):491--526, 2007.

\end{thebibliography}

\end{document}